\documentclass[12 pt, preprint]{elsarticle}
\usepackage{amsmath,amsthm}
\usepackage{amssymb}
\usepackage{enumerate}
\usepackage{graphicx}
\newtheorem{thm}{Theorem}[section]
\newtheorem{cor}[thm]{Corollary}
\newtheorem{lem}[thm]{Lemma}
\newtheorem{prop}[thm]{Proposition}
\newtheorem{rem}[thm]{Remark}
\theoremstyle{definition}

\newcommand{\Finp}{{\mathbb{F}_p}}
\newcommand{\Zet}{{\mathbb{Z}}}
\newcommand{\Nat}{{\mathbb{N}}}
\newcommand{\Per}{{\mathbb{S}}}
\numberwithin{equation}{section}
\begin{document}
\begin{frontmatter}
\baselineskip=17pt

\title{An inverse theorem in $\Finp$ and rainbow free colorings.}
\author{Mario  Huicochea\\
CINNMA\\ 
San Isidro 303 Juriquilla Fracc. Villas del Mes\'{o}n\\
 76226 Qro., M\'{e}xico\\
E-mail: dym@cimat.mx
}

\date{}







\begin{abstract}
Let $\Finp$ be the field with $p$ elements with $p$ prime, $X_1,\ldots, X_n$ pairwise disjoint subsets of $\Finp$ with at least $3$ elements  such that $\sum_{i=1}^n|X_i|\leq p-5$, and $\Per_n$ the set of permutations of $\{1,2,\ldots, n\}$. If $a_1,\ldots,a_n\in\Finp^*$ are not all equal, we characterize the subsets $X_1,\ldots, X_n$ which satisfy
\begin{equation*}
\Bigg|\bigcup_{\sigma\in\Per_n}\sum_{i=1}^na_{\sigma(i)}X_i\Bigg|\leq \sum_{i=1}^n|X_i|.
\end{equation*}
This result has the following application: For $n\geq 2$, $b\in\Finp$ and $a_1,\ldots, a_n$  as above,  we characterize the colorings  $\bigcup_{i=1}^nC_i=\Finp$ where each color has at least 3 elements such that  $\sum_{i=1}^na_ix_i=b$ has not rainbow solutions.
\end{abstract}

\end{frontmatter}
\section{Introduction}

In this article $p$ is a prime number, $\Finp$ the field with $p$ elements, $\Finp^*:=\Finp\setminus\{0\}$, $\Per_n$ is the set of permutations of $\{1,\ldots, n\}$, and  $[s]$ the greatest integer less than or equal to $s\in \mathbb{R}$.  Identifying $\Finp$ with $\Zet/p\Zet$, if $x\in\Zet$, then $\overline{x}$ is its image under the canonical projection $\Zet\rightarrow \Zet/p\Zet$. For $x,y\in\Finp$, define $[x,y]:=\big\{x,x+\overline{1},\ldots,x+\overline{i}\big\}$ where $i$ is the element of $\{0,1,\ldots, p-1\}\subseteq \Zet$  such that $\overline{i}=y-x$. For $r\in\Finp$ and $X\subseteq \Finp$, write $rX:=\{rx:\;x\in X\}$. Readily $X\subseteq \Finp$ is an arithmetic progression with common difference $r\in\Finp$ if and only if there are $x,y\in\Finp$ such that $X=r[x,y]$. An important and trivial fact  that will be used several times is the following
\begin{equation*}
r[x,y]=(-r)[-y,-x]\qquad \forall\;r,x,y\in\Finp.
\end{equation*}

Given $X$ and $Y$ subsets of $\Finp$, it is natural to ask whether $X$ and $Y$ have a particular structure when their sumset $X+Y$ is \emph{small}; the answers to this question are known as inverse theorems. Vosper [11] found the first non-trivial inverse theorem; also Hamidoune and R\o dseth [7] obtained an important inverse theorem with really few conditions on $|X|$ and $|Y|$, see Section 2 for the precise statement. Also for special subsets $X$ and $Y$ of $\Finp$ there exist interesting inverse theorems; for instance  Freiman [5] improved Vosper Theorem if $X=Y$, and Serra and Z\'{e}mor [10] generalized also Vosper Theorem. It is natural to ask whether we can generalize these results for arbitrarily many subsets $X_1,\ldots, X_n$ of $\Finp$;  Conlon [2] provided a generalization of Vosper and Hamidoune-R\o dseth Theorems for $n\geq 3$ when $\min_{1\leq i\leq n}|X_i|\geq n+1$, $\big|\sum_{i=1}^nX_i\big|\leq p-1$ and $p\geq 3n^2-4n-3$. The main result of this paper is the following inverse theorem.
\begin{thm}
\label{R1}
Let $n\geq 2$ and $X_1,\ldots, X_n$ be pairwise disjoint subsets of $\Finp$ such that $\min_{1\leq i\leq n}|X_i|\geq 3$ and $\sum_{i=1}^n|X_i|\leq p-5$. If $a_1,\ldots, a_n\in\Finp^*$  are not all equal, one of the following statements holds true:
\begin{enumerate}
\item[(i)] $n=2, a_1=-a_2$ and  $\{X_1,X_2\}=\{r[x,y],r([y+c,x-c]\setminus \{z\})\}$ for some $x,y,c,r,z\in\Finp$.

\item[(ii)] $|\bigcup_{\sigma\in \Per_n}\sum_{i=1}^na_{\sigma(i)}X_i|>\sum_{i=1}^n|X_i|$.
\end{enumerate}
\end{thm}
 
If $C_1,\ldots, C_n$ are pairwise disjoint subset of $\Finp$ such that $\bigcup_{i=1}^nC_i=\Finp$, we say $\mathcal{C}=\{C_i\}_{i=1}^n$ is a n-coloring of $\Finp$. Given a n-coloring $\mathcal{C}$ and an equation $\sum_{i=1}^na_ix_i=b$ with $a_1,\ldots,a_n\in\Finp^*$ and $b\in\Finp$, we say that  $\mathcal{C}$ is rainbow free with respect to this equation if $\sum_{i=1}^na_{\sigma(i)}v_{i}\neq b$ for all $\sigma\in\Per_n$  and $v_i\in C_i$. For $\lambda,\mu,x\in\Finp$, write
\begin{equation*}
S_{(\lambda,\mu)}(x):=\Bigg\{\lambda^k x+\Bigg(\sum_{i=0}^{k-1}\lambda^i\Bigg)\mu:\;k\in\Nat\Bigg\}.
\end{equation*}
 
 Jungi\'{c} \emph{et al.} [8] showed that the inverse theorems are  powerful tools to study the rainbow colorings. In the case where $n=3$, explicit characterizations of the equations  that have rainbow free  colorings are provided for example in [8], [9] and [6]. For arbitrary $n$ Conlon [2] showed that under the assumptions $\min_{1\leq i\leq n}|C_i|\geq n$ and $p\geq 3n^2-4n-3$, $\mathcal{C}$ is rainbow free with respect to $\sum_{i=1}^na_{i}v_{i}=b$ only if $a_1=\ldots=a_n$. As an application of Theorem \ref{R1}, we improve Conlon's lower bound taking $3$ instead of $n$ except in a very particular case; more precisely we show the following theorem.
 
\begin{thm}
\label{R2}
Let $n\geq 2$ and $\mathcal{C}:=\{C_i\}_{i=1}^n$ be a $n-$coloring of $\Finp$ with $|C_i|\geq 3$ for all $i\in\{1,\ldots,n\}$. If $a_1,\ldots, a_n\in\Finp^*$  are not all equal and $b\in\Finp$, $\mathcal{C}$ is rainbow free with respect to 
\begin{equation}
\label{E48}
\sum_{i=1}^na_ix_i=b
\end{equation}
if and only if the following conditions are satisfied
\begin{enumerate}
\item[(i)] $n=2$. 
 
\item[(ii)]  There are $z_1,\ldots, z_m,y_1\ldots, y_q\in\Finp$ such that
\begin{equation}
\label{E45}
C_1=\bigcup_{i=1}^m S_{\big(-a_2a_1^{-1},ba_1^{-1}\big)}(z_i)\qquad\text{and}\qquad C_2=\bigcup_{i=1}^q S_{\big(-a_2a_1^{-1},ba_1^{-1}\big)}(y_i).
\end{equation}
\end{enumerate}
\end{thm}

The paper is organized in the following way. In Section 2 we establish the main tools of Additive Number Theory that will be used in the following sections. In Section 3 we prove Theorem \ref{R1} for $n=2$ and in Section 4  we show it for $n=3$. To show Theorem \ref{R1}, we need to study some special cases when $n>3$ and this is done in Section 5. In Section 6 we complete the proof of Theorem \ref{R1} and  we conclude the proof of Theorem \ref{R2} in Section 7.

\section{Preliminaries}
First we recall some important  results
\begin{thm}\emph{(Cauchy-Davenport)}
\label{R3}
If $X_1,X_2\subseteq \Finp$ are not empty, then
\begin{equation*}
|X_1+X_2|\geq \min\{p,|X_1|+|X_2|-1\}.
\end{equation*}
\end{thm}
\begin{proof}
See [3] and [4].
\end{proof}

\begin{thm}\emph{(Vosper)}
\label{R4}
Let $X_1$ and $X_2$ be subsets of $\Finp$ such that
\begin{equation*}
  \min\{|X_1|,|X_2|\}\geq 2\qquad  \text{and}\qquad  |X_1|+|X_2|-1=|X_1+X_2|\leq p-2.
 \end{equation*}  Then there are $x_1,x_2,y_1,y_2,r\in\Finp$ such that 
$X_1=r[x_1,y_1]$ and $X_2=r[x_2,y_2]$.
\end{thm}
\begin{proof}
See [11].
\end{proof}

\begin{thm}\emph{(Hamidoune-R\o dseth)}
\label{R5}
Let $X_1$ and $X_2$ be subsets of $\Finp$ such that
\begin{equation*}
  \min\{|X_1|,|X_2|\}\geq 3\qquad  \text{and}\qquad  7\leq |X_1|+|X_2|=|X_1+X_2|\leq p-4.
 \end{equation*}  Then there are $x_1,x_2,y_1,y_2,z_1,z_2,r\in\Finp$ such that $X_1=r([x_1,y_1]\setminus\{z_1\})$  and $X_2=r([x_2,y_2]\setminus\{z_2\})$.
\end{thm}
\begin{proof}
See [7].
\end{proof}

\begin{prop}
\label{R6}
Assume that $p\geq 11$. If $X_1,X_2\subseteq \Finp$  are such that $|X_1|=|X_2|=3$ and $|X_1+X_2|\leq 6$, then one of the following holds true:
\begin{enumerate}
\item[(i)] $X_1=X_2+z$ for some $z\in\Finp$.

\item[(ii)] $\{X_1,X_2\}=\{r[x_1,x_1+2],r([x_2,x_2+1]\cup\{x_2+3\})\}$ for some $r,x_1,x_2\in\Finp$.
\end{enumerate}
\end{prop}
\begin{proof}
See [6, Lemma 27].
 \end{proof}
 
 \begin{lem}
\label{R25}
Let $x,y,x',y'$ be elements of $\Finp$, $r,r'\in \Finp^*$ and $Y$ a subset of $[x',y']$. Call $R$ the element of $\{0,\ldots, p-1\}\subseteq\Zet$ such that $\overline{R}=r^{-1}r'$. If $r[x,y]=r'([x',y']\setminus Y)$, then
\begin{equation*}
|Y|+1\geq \min\big\{p-R,R,|[x,y]|\big\}.
\end{equation*}
Furthermore, if $|[x,y]|\geq \min\{R,p-R\}=:k$, then 
\begin{equation*}
|[x',y']|\geq \bigg[\frac{|[x,y]|}{k}\bigg]+(k-1)\bigg[\frac{p}{k}\bigg].
\end{equation*}
\end{lem}
\begin{proof}
For the first claim, note that 
\begin{align*}
\Big|\big(r'[x',y']\setminus Y)+r'\big)\setminus r'([x',y']\setminus Y)\Big|&=\Big|\big(([x',y']\setminus Y)+1\big)\setminus ([x',y']\setminus Y)\Big|\\
&\leq |Y|+1.
\end{align*}
On the other hand 
\begin{align*}
\Big|\big(r[x,y]+r'\big)\setminus r[x,y]\Big|&=\Big|\big([x,y]+r^{-1}r'\big)\setminus [x,y]\Big|\\
&\geq \min\big\{p-R,R,|[x,y]|\big\}
\end{align*}
and these inequalities show the first claim.

For the second statement, assume without loss of generality that $k=R$. Let $k'$ be the element of $\{0,\ldots,p-1\}\subseteq \Zet$ such that $\overline{k'}=r{r'}^{-1}$. Then there is $s\in\{0,\ldots,p-1\}\subseteq \Zet$ such $kk'=sp+1$. If $u\in\{1,\ldots,\big[\frac{p}{k}\big]\}\subseteq \Zet$, then $\overline{(uk)k'}=\overline{u}$ since $\overline{kk'}=1$. In particular,  for all $i,j\in\{0,\ldots, k-1\}$ and $z\in\Finp$, we have that $irr'^{-1},jrr'^{-1}\in \Big[z,z+\overline{\big[\frac{p}{k}\big]}-1\Big]$ only if $i=j$. This implies straightforward by the Pigeonhole principle that 
\begin{align*}
\min\Big\{\big|\big[x'',y''\big]\big|:\;rr'^{-1}[x,y]\subseteq \big[x'',y''\big]\Big\}&\geq \bigg[\frac{|[x,y]|}{k}\bigg]+(k-1)\bigg[\frac{p}{k}\bigg]
\end{align*} and the claim follows.
\end{proof}

\begin{prop}
\label{R7}
Let  $X$ be a  subset of $\Finp$.  
\begin{enumerate}
\item[(i)] Assume  that $3\leq |X|\leq p-5$ and $X=r([x,y]\setminus\{z\})$ for some $x,y,z,r\in\Finp$. If there are $x',y',z',r'\in\Finp$ such that $X=r'([x',y']\setminus\{z'\})$, then $r'\in\{\pm r\}$. 

\item[(ii)] Assume that $3\leq |X|\leq p-5$ and $X=r([x,y]\setminus\{z,z'\})$ for some $x,y,r,z,z'\\\in\Finp$ such that $z\neq z'$ and $z,z'\in[x+1,y-1]$. If there are $x',y',r'\in\Finp$   such that $X=r'[x',y']$, then $|X|=3$ and $r'\in\{\pm 2r\}$.
\end{enumerate}
\end{prop}
\begin{proof}
We have that (i) is a  consequence of [6, Lemma 16] up to some cases which are solved easily. Then  (ii) is a straightforward  consequence of Lemma \ref{R25}.
\end{proof}
The following result is an application of Proposition \ref{R7}.
\begin{cor}
\label{R14}
Let $x_1,x_2,y_1,y_2\in\Finp$ be such that 
\begin{equation*}
4\leq |[x_2,y_2]|\leq p-5\qquad\text{and}\qquad 0\leq |[x_1,y_1]|-|[x_2,y_2]|\leq 2.
\end{equation*}
If $r_1,r_2\in\Finp^*$ satisfy that $r_2[x_2,y_2]\subseteq r_1[x_1,y_1]$, then $r_1\in\{\pm r_2\}$.
\end{cor}
\begin{lem}
\label{R8}
Let $\lambda\in\Finp$ be such that $\lambda^4+\lambda^2+1=0$ and 
\begin{equation*}
X:=\{0,1,2,\lambda^2+1,\lambda^2+2,2\lambda^2+2\}.
\end{equation*}
 If $[y,y+2]\subseteq \lambda X$, then $p\leq 7$ or
 \begin{equation*}
y= \left\{ \begin{array}{lll}
2 \lambda & \mbox{if $\lambda^3=1$} \\
2\lambda-2 & \mbox{if $\lambda^3=-1$}.\end{array} \right.
 \end{equation*}
  \end{lem}
\begin{proof}
See [6, Lemma 29].
\end{proof}

\begin{prop}
\label{R23}
Let  $x_1,\ldots, x_n,y_1,\ldots, y_n$ and $r$ be elements of $\Finp$ such that
 $[x_i,y_i]\neq \emptyset$ for all $i\in\{1,\ldots,n\}$ and $r[x_i,y_i]\cap r[x_j-1,y_j+1]\neq \emptyset$ only if $i=j$.  If $x,y\in\Finp$ are such that $|[x,y]|\geq 2$ and $X:=\bigcup_{i=1}^nr[x_i,y_i]$, then 
\begin{equation*}
|X+r[x,y]|\geq \min\{p,|X|+|[x,y]|+n-2\}.
\end{equation*}
\end{prop}
\begin{proof}
First assume that for all $x',y'\in\Finp$ such that $r[x',y']\subseteq \Finp\setminus X $ we have $|[x,y]|>|[x',y']|$; then $X+r[x,y]=\Finp$ and this implies the claim. Now assume that there are  $x',y'\in\Finp$ such that $r[x',y']\subseteq \Finp\setminus X $ and  $|[x,y]|\leq |[x',y']|$; without loss of generality suppose that $x'=y_1+1$ and $y'=x_2-1$, and set $Y:=\bigcup_{i=2}^n(r[x_i,y_i]+r[x,x+1])$. Hence
\begin{align*}
|X+r[x,y]|&\geq|r[x_1+x,y_1+y]|+|Y|&\text{since }\\
&&r[x',y']\subseteq \Finp\setminus X\\
&=|r[x_1,y_1]|+|r[x,y]|-1+|Y|\\
&\geq |r[x_1,y_1]|+|[x,y]|-1+\sum_{i=2}^n(|r[x_i,y_i]|+1)\\
&=|X|+|[x,y]|+n-2.
\end{align*}
\end{proof}
This lower bound of $|X+r[x,y]|$ will be used in the following result.
\begin{lem}
\label{R12}
Let $x_1,y_1,x_2,y_2\in\Finp$ be such that 
\begin{equation*}
3\leq\min\{|[x_1,y_1]|,|[x_2,y_2]|\}\leq\max\{|[x_1,y_1]|,|[x_2,y_2]|\}\leq p-6.
\end{equation*}
If $r_1,r_2\in\Finp^*$ satisfy the inequality
\begin{equation}
\label{E54}
|r_1[x_1,y_1]+r_2[x_2,y_2]|\leq |[x_1,y_1]|+|[x_2,y_2]|+1,
\end{equation}
then  one of the following statements holds true
\begin{enumerate}
\item[(i)] $r_1\in\{\pm r_2\}$.
\item[(ii)] $r_2\in \{\pm 2r_1\}$ and $|[x_2,y_2]|=3$.
\item[(iii)] $r_1\in \{\pm 2r_2\}$ and $|[x_1,y_1]|=3$.
\end{enumerate}
Furthermore, if (\ref{E54}) is strict, then $r_1\in\{\pm r_2\}$. 
\end{lem}
\begin{proof}
Write $S:=r_1[x_1,y_1]+r_2[x_2,y_2]$. By Theorem \ref{R3} we have to work only with 3 cases:
\begin{enumerate}
\item[]If $|S|=|[x_1,y_1]|+|[x_2,y_2]|-1$, then $r_1\in\pm\{r_2\}$ by Theorem \ref{R4} and Proposition \ref{R7}.

\item[]If $|S|=|[x_1,y_1]|+|[x_2,y_2]|$, then $r_2[x_2,y_2]=r_1[x_1^*,y_1^*]\cup r_1[x_1^{**},y_1^{**}]$ by Proposition \ref{R23}. Furthermore, since $|S|=|[x_1,y_1]|+|[x_2,y_2]|$, it can be seen  that $r_2[x_2,y_2]=r_1([x_1',y_1']\setminus\{z'\})$ for some $x_1',y_1'\in\Finp$ and $z'\in\big[x_1',y_1'\big]$; however, $r_2[x_2,y_2]$ cannot have this shape by Proposition \ref{R7}.

\item[]If $|S|=|[x_1,y_1]|+|[x_2,y_2]|+1$, then $r_2[x_2,y_2]=r_1[x_1^*,y_1^*]\cup r_1[x_1^{**},y_1^{**}]\cup r_1[x_1^{***},y_1^{***}]$ for some $x_1^*,y_1^*,x_1^{**},y_1^{**},x_1^{***},y_1^{***}\in\Finp$ by Proposition \ref{R23}. If $|[x_1,y_1]|>3$, then $r_2[x_2,y_2]=r_1[x_1',y_1']\setminus\{z,z'\}$ for some $x_1',y_1',z,z'\in\Finp$ with $z,z'\in[x_1'+1,y_1'-1]$ and $z\neq z'$. By Proposition \ref{R7} this means that $|r_2[x_2,y_2]|=3$ and $r_2\in\{\pm 2r_1\}$.  If $|[x_1,y_1]|=3$ and $|[x_2,y_2]|>3$, then we proceed as above. Hence it remains the case $|[x_1,y_1]|=|[x_2,y_2]|=3$; under this assumption, $r_2[x_2,y_2]$ is as above or  $r_2[x_2,y_2]=r_1[x_1',y_1']\cup r_1[x_1'',y_1'']$ with $r_1[x_1',y_1']\cap r_1[x_1''-2,y_1''+2]=\emptyset$ and it is straightforward to check that $r_1,r_2$ are as in (ii) or (iii). 
\end{enumerate}
The second claim is proven above.
\end{proof}
\begin{lem}
\label{R22}
Assume that $p\geq 5$. Let  $X$ be a subset of $\Finp$ with $|X|=3$ and $w\in \Finp$  such that  \begin{equation}
\label{E55}
X+X+w=-(X+X+w).
\end{equation}
Then there is $r\in\Finp$ such that $X=r[-1,1]-2^{-1}w$. In particular, $|X+X+w|=|X|+|X|-1$.
\end{lem}
\begin{proof}
Write $X':=X+2^{-1}w=\{x',y',z'\}$ so (\ref{E55}) is equivalent to 
\begin{equation*}
X'+X'=-(X'+X');
\end{equation*}
in particular 
\begin{equation*}
\sum_{w'\in X'+X'}w'=\sum_{w'\in X'+X'}-w'=\sum_{-w'\in X'+X'}w'
\end{equation*}
so $x'+y'+z'=0$ and $X'+X'=\{2x',2y',x'+y',-y',-x',-2x'-2y'\}$. Since $2x'\in -(X'+X')$, we conclude that $X'=r[-1,1]$ for some $r\in\Finp$ analyzing all the possible values of $2x'$.
\end{proof}
\section{Case $n=2$}
\begin{lem}
\label{R9}
Let $X_1,X_2\subseteq \Finp$ be disjoint subsets such that $|X_1|=|X_2|=3$ and $|X_1|+|X_2|\leq p-5$. If $a_1, a_2$ are elements in $\Finp^*$ such that $a_1\neq a_2$
\begin{equation*}
|a_1X_1+a_2X_2\cup a_1X_2+a_2X_1|\leq |X_1|+|X_2|,
\end{equation*} then $a_1=-a_2$ and there exist $x,y,c,r,z\in\Finp$ such that
\begin{equation*}
\{X_1,X_2\}=\big\{r[x,y],r([y+c,x-c]\setminus\{z\})\big\}. 
\end{equation*}
\end{lem}
\begin{proof}
Write $S:=a_1X_1+a_2X_2\cup a_1X_2+a_2X_1$. If $|a_1X_1+a_2X_2|=|X_1|+|X_2|-1$, then $a_1X_1=r[x_1,x_1+2]$ and $a_2X_2=r[x_2,x_2+2]$ for some $r,x_1,x_2\in\Finp$ by Theorem \ref{R4}. Thus  $a_2X_1=a_2a_1^{-1}r[x_1,x_1+2]$ and $a_1X_2=a_1a_2^{-1}r[x_2,x_2+2]$. On the other hand, the inequality 
\begin{equation*}
|a_1X_2+a_2X_1|\leq |X_1|+|X_2|
\end{equation*}
implies that $a_2a_1^{-1}\in\{\pm a_1a_2^{-1}\}$ by Lemma \ref{R12}. Hence
\begin{align}
S&\in\{r[x_1+x_2,x_1+x_2+4]\cup a_2a_1^{-1}r[x_1+x_2,x_1+x_2+4],\nonumber\\
&\qquad r[x_1+x_2,x_1+x_2+4]\cup a_2a_1^{-1}r[x_1-x_2-2,x_1-x_2+2]\}.\label{E49}
\end{align}
Since $|S|\leq |X_1|+|X_2|$, there are $z'\in a_1X_1+a_2X_2$ and $z''\in a_2X_1+a_1X_2$ such that
 \begin{equation*}
  (a_1X_1+a_2X_2)\setminus\{z'\}=(a_2X_1+a_1X_2)\setminus\{z''\}.
 \end{equation*}
Then (\ref{E49}) implies that  $a_2a_1^{-1}=-1$ by Proposition \ref{R7}. The equality $a_1=-a_2$ and (\ref{E49}) yield  that $x_1+x_2=-2$ and consequently  $X_1\cap X_2\neq\emptyset$. If $|a_1X_2+a_2X_1|=|X_1|+|X_2|-1$, then we proceed as above so from now on we assume that
\begin{equation}
\label{E56}
|a_1X_1+a_2X_2|=|a_1X_2+a_2X_1|=|X_1|+|X_2|.
\end{equation}
By Proposition \ref{R6} we have the following cases:
\begin{enumerate}
\item[]Either there is not $w\in\Finp$ such that $a_1X_1=a_2X_2+w$ or there is not $w\in\Finp$ such that $a_2X_1=a_1X_2+w$. Assume without loss of generality that there is not $w\in\Finp$ such that $a_1X_1=a_2X_2+w$; from Proposition \ref{R6} we may assume that $a_1X_1=r[x_1,x_1+2]$ and $a_2X_2=r([x_2,x_2+1]\cup\{x_2+3\})$ for some $r,x_1,x_2\in\Finp$ (the other cases are solved in the same way). Hence $a_2X_1=a_2a_1^{-1}r[x_1,x_1+2]$, $a_1X_2=a_1a_2^{-1}r([x_2,x_2+1]\cup\{x_2+3\})$,
 and therefore (\ref{E56}) tell us that  $a_2a_1^{-1}\in\{\pm a_1a_2^{-1}\}$  by Proposition \ref{R6} and Proposition \ref{R7}. Consequently 
\begin{align}
S&\in\{r[x_1+x_2,x_1+x_2+5]\cup a_2a_1^{-1}r[x_1+x_2,x_1+x_2+5],\nonumber\\
&\qquad r[x_1+x_2,x_1+x_2+5]\cup a_2a_1^{-1}r[x_1-x_2-3,x_1-x_2+2]\},\label{E50}
\end{align}
and, since $|S|\leq |X_1|+|X_2|$, we have that
\begin{equation*}
r[x_1+x_2,x_1+x_2+5]= a_2a_1^{-1}r[x_1+x_2,x_1+x_2+5]
\end{equation*}
or 
\begin{equation*}
r[x_1+x_2,x_1+x_2+5]= a_2a_1^{-1}r[x_1-x_2-3,x_1-x_2+2];
\end{equation*}
 in any case $a_2a_1^{-1}=-1$ by Proposition \ref{R7}. The equality $a_1=-a_2$ and (\ref{E50}) yield the solution.
  
\item[]Now we proceed in the case where there are $w_1,w_2\in\Finp$ such that $a_1X_1=a_2X_2+w_1$ and $a_2X_1=a_1X_2+w_2$. We have
\begin{equation*}
a_2^2a_1^{-1}X_2+a_2a_1^{-1}w_1=a_2X_1=a_1X_2+w_2
\end{equation*}
so 
\begin{equation*}
a_2^2a_1^{-2}X_2+a_2a_1^{-2}w_1-a_1^{-1}w_2=X_2.
\end{equation*}
Set $\lambda:=a_2a_1^{-1}$ and $\mu:=a_2a_1^{-2}w_1-a_1^{-1}w_2$ so $X_2=\lambda^2X_2+\mu$. If $\lambda=-1$, then $\mu=0$ and thereby $a_1+a_2=w_1+w_2=0$; however, this is impossible since (\ref{E56}) would contradict Lemma \ref{R22}. From now on suppose that  $\lambda^2\neq 1$. If $\lambda^2=-1$, then  $X_2$ is an arithmetic progression since $|X_2|=3$; consequently $a_1X_2$ and $a_2X_1$ are arithmetic progressions with the same difference and hence 
\begin{equation*}
|a_1X_2+a_2X_1|=|X_1|+|X_2|-1
\end{equation*}
contradicting (\ref{E56}). From now on suppose that $\lambda^2\neq -1$ and write $X_2:=\{x_2,y_2,z_2\}$. If $\lambda^2 w+\mu=w$ for some $w\in X_2$, then for all $w'\in X_2\setminus\{w\}$
\begin{equation}
\label{E57}
\lambda^2w'+\mu\neq w';
\end{equation}
however, since $|X_2|=3$, we get that
\begin{equation}
\label{E58}
\lambda^4w'+(\lambda^2+1)\mu=w'.
\end{equation}
Thus (\ref{E58}) implies  that  (\ref{E57}) is false insomuch as $\lambda^2+1\neq 0$.  Then without loss of generality $y_2=\lambda^2 x_2+\mu$, $z_2=\lambda^2 y_2+\mu$ and $x_2=\lambda^2 z_2+\mu$; particularly $1+\lambda^2+\lambda^4=0$. From (\ref{E56}) we see that $a_1X_1+a_2X_2=a_2X_1+a_1X_2$ so
\begin{equation}
\label{E16}
a_2a_1^{-1}X_2+a_2a_1^{-1}X_2=X_2+X_2-(w_1-w_2)a_1^{-1}.
\end{equation}
Adding $-2x_2\lambda$  and multiplying (\ref{E16}) by $\theta:=((\lambda^2-1)x_2+\mu)^{-1}$ ,  we obtain that
\begin{align}
\label{E17}
\lambda \{0,1,2,\lambda^2+1,\lambda^2+2,2\lambda^2+2\}&=\{0,1,2,\lambda^2+1,\lambda^2+2,2\lambda^2+2\}\nonumber\\
&\quad +((w_2-w_1)a_1^{-1}+2x_2(1-\lambda))\theta.
\end{align}
By Lemma \ref{R8} we conclude that
\begin{equation}
\label{E51}
\lambda^3=1\qquad \text{and}\qquad 2\lambda=((w_2-w_1)a_1^{-1}+2x_2(1-\lambda))\theta
\end{equation}
or 
\begin{equation}
\label{E52}
\lambda^3=-1\qquad \text{and}\qquad 2\lambda-2=((w_2-w_1)a_1^{-1}+2x_2(1-\lambda))\theta.
\end{equation}
If (\ref{E51}) is true, then $2\lambda\mu=(w_2-w_1)a_1^{-1}$ and thereby 
\begin{equation}
\label{E53}
w_2(1+2\lambda)=w_1(1+2\lambda^2); 
\end{equation}
on the other hand 
\begin{equation*}
X_2=\lambda^2X_2+\mu=\lambda^4 X_2+(\lambda^2+1)\mu=\lambda X_2+(\lambda^2+1)\mu
\end{equation*}
and by assumption $X_1=\lambda X_2+a_1^{-1}w_1$; however, we get from (\ref{E53}) that
$(\lambda^2+1)\mu=a_1^{-1}w_1$ contradicting the disjointedness of $X_1$ and $X_2$. If (\ref{E52}) is true, then (\ref{E17}) implies that $
\{0,\lambda,\lambda-1\}=\{3\lambda-2,3\lambda-1,4\lambda-2\}$
which is impossible.
\end{enumerate}
\end{proof}
\begin{lem}
\label{R10}
Let $X_1,X_2\subseteq \Finp$ be disjoint subsets such that $\min\{|X_1|,|X_2|\}\geq 3$. If 
\begin{equation*}
|X_1-X_2\cup X_2-X_1|\leq |X_1|+|X_2|\leq p-4,
\end{equation*}
then  $\{X_1,X_2\}=\big\{r[x,y],r([y+c,x-c]\setminus\{z\})\big\}$ for some $x,y,c,r,z\in\Finp$.
\end{lem}
\begin{proof}
Write $S:=X_1-X_2\cup X_2-X_1$. By Lemma \ref{R9} we may assume that $\max\{|X_1|,|X_2|\}>3$ from now on. By Theorem \ref{R3}
\begin{equation*}
|X_1|+|X_2|-1\leq\min\{|X_1-X_2|, |X_2-X_1|\}
\end{equation*}
so $|X_1|+|X_2|-1\leq|S|$. If  $|X_1|+|X_2|-1=|S|$, then $S=X_1-X_2=X_2-X_1$. In the case where $|X_1|+|X_2|-1=|X_1-X_2|$ Theorem \ref{R4} establishes that there exist $x_1,x_2,y_1,y_2,r\in\Finp$ such that $X_1=r[x_1,y_1]$ and $X_2=r[x_2,y_2]$.  Hence $X_1-X_2=X_2-X_1$ if and only if there is $c\in \Finp$ such that $X_1=r[x_1,y_1]$ and  $X_2=r[y_1+c,x_1-c]$. We suppose that $|X_1|+|X_2|=|S|$ from now on and we have to study two cases:
\begin{enumerate}
\item[] Assume that $|X_1|+|X_2|-1\not\in\{|X_1-X_2|,|X_2-X_1|\}$. We know that $7\leq |X_1|+|X_2|$.  On one hand $|X_1-X_2|=|X_1|+|X_2|$, 
and by Theorem \ref{R5} there are $x_1,x_2,y_1,y_2,r\in \Finp$ such that
$X_1\subseteq r[x_1,y_1]$ and $X_2\subseteq r[x_2,y_2]$  with $|X_1|+1=|[x_1,y_1]|$ and $|X_2|+1=|[x_2,y_2]|$. On the other hand  $|X_2-X_1|=|X_1|+|X_2|$,
and by Theorem \ref{R5} there are $x'_1,x'_2,y'_1,y'_2,r'\in \Finp$ such that
$X_1\subseteq r'[x'_1,y'_1]$ and $ X_2\subseteq r'[x'_2,y'_2]$ with $|X_1|+1=|[x'_1,y'_1]|$ and $|X_2|+1=|[x'_2,y'_2]|$. By Proposition \ref{R7} $r\in\{\pm r'\}$; assume without loss of generality that $r=r'$. Hence $|(X_1-X_2)\cup (X_2-X_1)|=|X_1|+|X_2|$ if and only if there  are $x,y,c,z\in\Finp$ with $z\in[y+c,x-c]$ such that $\{X_1,X_2\}=\{r[x,y],r([y+c,x-c]\setminus\{z\})\}$.

\item[] For the remaining case,  assume without loss of generality that $|X_1|+|X_2|-1=|X_1-X_2|$. By Theorem \ref{R4} there are $x_1,x_2,y_1,y_2,r\in \Finp$ such that $X_1=r[x_1,y_1]$ and $ X_2=r[x_2,y_2]$. Inasmuch as $|S|=|X_1|+|X_2|$, we conclude that $X_1-X_2\neq X_2-X_1$, and furthermore
\begin{equation}
\label{E3}
|X_1-X_2\cap X_2-X_1|=|X_1|+|X_2|-1.
\end{equation}
Finally, (\ref{E3}) let us state that there are $x,y,c\in\Finp$ such that 
\begin{equation*}
\{X_1,X_2\}=\{r[x,y],r[y+c,x-c-1]\}.
\end{equation*}
\end{enumerate}
\end{proof}

\begin{lem}
\label{R11}
Let $X_1$ and $X_2$ be disjoint subsets of $\Finp$ such that $\min\{|X_1|,|X_2|\}\geq 3$ and $a_1,a_2\in\Finp^*$ with $a_1\not\in\{\pm a_2\}$. If $|X_1|+|X_2|\leq p-5$ , then 
\begin{equation*}
|a_1X_1+a_2X_2\cup a_1X_2+a_2X_1|>|X_1|+|X_2|.
\end{equation*} 
\end{lem}
\begin{proof}
 From Lemma \ref{R9} we may assume that $\max\{|X_1|,|X_2|\}>3$ from now on. Set $S:=a_1X_1+a_2X_2\cup a_1X_2+a_2X_1$. We assume that the lemma is false and we obtain a contradiction. By Theorem \ref{R3}
\begin{equation*}
 |X_1|+|X_2|-1\leq |a_1X_1+a_2X_2|,|a_2X_1+a_1X_2|\leq |X_1|+|X_2|.
\end{equation*}
Thus Theorem \ref{R4} and Theorem \ref{R5} imply the existence of $x_1,x_2,y_1,y_2,r\in\Finp$ such that $a_1X_1\subseteq r[x_1,y_1]$ and $a_2X_2\subseteq r[x_2,y_2]$ with $|[x_1,y_1]|=|X_1|+1$ and $|[x_2,y_2]|=|X_2|+1$. In the same way, there are $x'_1,x'_2,y'_1,y'_2,r'\in\Finp$ such that  $a_2X_1\subseteq r'[x'_1,y'_1]$ and $a_1X_2\subseteq r'[x'_2,y'_2]$ with $|[x'_1,y'_1]|=|X_1|+1$ and $|[x'_2,y'_2]|=|X_2|+1$. Since $a_2X_1=a_2a_1^{-1}(a_1X_1)\subseteq a_2a_1^{-1}r[x_1,y_1]$, we get that $a_2a_1^{-1}r\in \{\pm r'\}$ by Proposition \ref{R7}. Thus there are  $x_3,y_3,x'_3,y'_3\in\Finp$ such that $a_1X_1+a_2X_2\subseteq r[x_3,y_3]$ and $a_2X_1+a_1X_2\subseteq a_2a_1^{-1}r[x'_3,y'_3]$ with $|[x_3,y_3]|=|[x'_3,y'_3]|=|X_1|+|Y_1|$. Then Proposition \ref{R7} and Corollary \ref{R14} yield $r\in\{\pm a_2a_1^{-1}r\}$ contradicting the assumption $a_1\not\in\{\pm a_2\}$.
 \end{proof}
\section{Case $n=3$}
\begin{lem}
\label{R13}
Let $X_1,X_2,X_3$ be pairwise disjoint subsets of $\Finp$ such that \\$\min_{1\leq i\leq 3} |X_i|\geq 3$ and $\sum_{i=1}^3|X_i|\leq p-3$. Assume that $a_1,a_2,a_3\in\Finp^*$ are such that there exist $a_i,a_j$ with $a_i\not\in\{\pm a_j\}$. Then
\begin{equation*}
\Bigg|\bigcup_{\sigma\in \Per_3}\sum_{i=1}^3a_{\sigma(i)}X_i\Bigg|> \sum_{i=1}^3|X_i|.
\end{equation*}
\end{lem}
\begin{proof}
Assume without loss of generality that $a_1\not\in\{\pm a_2\}$ and  set\\ $S :=\bigcup_{\sigma\in \Per_3}\sum_{i=1}^3a_{\sigma(i)}X_i$. We assume that $|S|\leq |X_1|+|X_2|+|X_3|$, and we arrive to a contradiction. By Lemma \ref{R11} 
\begin{equation*}
|a_1X_1+a_2X_2\cup a_1X_2+a_2X_1|>|X_1|+|X_2|.
\end{equation*}
Thus
\begin{align*}
\sum_{i=1}^3|X_i|&\leq |(a_1X_1+a_2X_2\cup a_1X_2+a_2X_1)|+|a_3X_3|-1\\
&\leq |(a_1X_1+a_2X_2\cup a_1X_2+a_2X_1)+a_3X_3|& \text{by Theorem \ref{R3}}\\
&=|a_1X_1+a_2X_2+a_3X_3\cup a_1X_2+a_2X_1+a_3X_3|\\
&\leq |S|;
\end{align*}
 particularly
\begin{equation}
\label{E18}
|S|=|(a_1X_1+a_2X_2\cup a_1X_2+a_2X_1)+a_3X_3|=|X_1|+|X_2|+|X_3|.
\end{equation}
From Theorem \ref{R4} and (\ref{E18}), there exist $x_3,y_3,x_3',y_3',r_3\in\Finp$ such that
\begin{equation}
\label{E19}
a_1X_1+a_2X_2\cup a_1X_2+a_2X_1=r_3[x_3',y_3']\qquad\text{and}\qquad a_3X_3=r_3[x_3,y_3].
\end{equation}
  In the same way, there are $x_1,y_1,x_2,y_2,r_1,r_2\in\Finp$ such that $a_3X_2=r_2[x_2,y_2]$ and $a_3X_1=r_1[x_1,y_1]$. Then
\begin{equation*}
  a_1X_1+a_2X_2=a_1a_3^{-1}r_1[x_1,y_1]+a_2a_3^{-1}r_2[x_2,y_2]
 \end{equation*}
 and 
\begin{equation*}
  a_2X_1+a_1X_2=a_2a_3^{-1}r_1[x_1,y_1]+a_1a_3^{-1}r_2[x_2,y_2]
 \end{equation*}
so, by Lemma \ref{R12} and (\ref{E18}), we have  the following cases:
\begin{enumerate}
\item[] If $a_1a_3^{-1}r_1\in\{\pm a_2a_3^{-1}r_2\}$ and $a_2a_3^{-1}r_1\in\{\pm a_1a_3^{-1}r_2\}$, then there are $z_1, z_2,\\ z'_1, z'_2\in\Finp$ such that 
 $a_1X_1+a_2X_2=a_1a_3^{-1}r_1[z_1,z_2]$ and $a_2X_1+a_1X_2=a_2a_3^{-1}r_1[z'_1,z'_2]$.
By Corollary \ref{R14} and (\ref{E19}), $a_1a_3^{-1}r_1,a_2a_3^{-1}r_1\in\{\pm r_3\}$ so $a_1\in\{\pm a_2\}$ contradicting the hypothesis.

\item[] If $a_1a_3^{-1}r_1\in\{\pm a_2a_3^{-1}r_2\}$ and $a_2a_3^{-1}r_1\not\in\{\pm a_1a_3^{-1}r_2\}$, then either $a_2a_3^{-1}r_1\in\{\pm 2a_1a_3^{-1}r_2\}$ or $a_1a_3^{-1}r_2\in\{\pm 2a_2a_3^{-1}r_1\}$. It will be assumed without loss of generality that $a_2a_3^{-1}r_1\in\{\pm 2a_1a_3^{-1}r_2\}$. Hence there are $z_1,z_2,z'_1,z'_2\in\Finp$ such that $a_1X_1+a_2X_2=a_2a_3^{-1}r_2[z_1,z_2]$ and $a_2X_1+a_1X_2=a_1a_3^{-1}r_2[z'_1,z'_2]$.
From Corollary \ref{R14} and (\ref{E19}), $a_2a_3^{-1}r_2,a_1a_3^{-1}r_2\in\{\pm r_3\}$ so $a_1\in\{\pm a_2\}$ which contradicts the assumption. The case $a_2a_3^{-1}r_1\in\{\pm a_1a_3^{-1}r_2\}$ and $a_1a_3^{-1}r_1\not\in\{\pm a_2a_3^{-1}r_2\}$ is solved in the same way.

\item[] Assume that  $a_1a_3^{-1}r_1\not\in\{\pm a_2a_3^{-1}r_2\}$ and $a_2a_3^{-1}r_1\not\in\{\pm a_1a_3^{-1}r_2\}$. Lemma \ref{R12} establishes that $a_1a_3^{-1}r_1\in\{\pm 2a_2a_3^{-1}r_2\}$ or $a_2a_3^{-1}r_2 \in \{\pm 2a_1a_3^{-1}r_1\}$,  and $a_2a_3^{-1}r_1\in\{\pm 2a_1a_3^{-1}r_2\}$ or $a_1a_3^{-1}r_2\in\{\pm 2a_2a_3^{-1}r_1\}$. In some of the cases, we arrive to a contradiction proceeding exactly as above. Up to symmetric cases, the unique possibility remaining is $a_1a_3^{-1}r_1\in\{\pm 2a_2a_3^{-1}r_2\}$ and $a_1a_3^{-1}r_2\in\{\pm 2a_2a_3^{-1}r_1\}$ with $[x_1,x_1+2]=[x_1,y_1]$ and $[x_2,x_2+2]=[x_2,y_2]$. Suppose without loss of generality that $a_1a_3^{-1}r_1=2a_2a_3^{-1}r_2$. If $a_1a_3^{-1}r_2=2a_2a_3^{-1}r_1$, then 
\begin{equation*}
a_1X_1+a_2X_2=a_2a_3^{-1}r_2[2x_1+x_2,2x_1+x_2+6]
\end{equation*} 
and 
\begin{equation*}
a_2X_1+a_1X_2=a_2a_3^{-1}r_1[x_1+2x_2,x_1+2x_2+6]
\end{equation*} so Corollary \ref{R14} and (\ref{E19}) establish that $r_1\in\{\pm r_2\}$; moreover, (\ref{E19}) leads to \begin{equation*}
a_2a_3^{-1}r_2[2x_1+x_2,2x_1+x_2+6]=a_2a_3^{-1}r_1[x_1+2x_2,x_1+2x_2+6]
\end{equation*} and therefore $X_1\cap X_2\neq \emptyset$ which is impossible. If $a_1a_3^{-1}r_2=-2a_2a_3^{-1}r_1$, then 
\begin{equation*}
a_1X_1+a_2X_2=a_2a_3^{-1}r_2[2x_1+x_2,2x_1+x_2+6]
\end{equation*} 
and 
\begin{equation*}
a_2X_1+a_1X_2=a_2a_3^{-1}r_1[x_1-2x_2-4,x_1-2x_2+2].
\end{equation*}
As above
 \begin{equation*}
a_2a_3^{-1}r_2[2x_1+x_2,2x_1+x_2+6]=a_2a_3^{-1}r_1[x_1-2x_2-4,x_1-2x_2+2]
\end{equation*} so  $r_1\in\{\pm r_2\}$  by  Proposition \ref{R7}; however, this contradicts the equalities $a_1a_3^{-1}r_1=2a_2a_3^{-1}r_2$ and $a_1a_3^{-1}r_2=-2a_2a_3^{-1}r_1$.
\end{enumerate}
\end{proof}
\begin{lem}
\label{R15}
Let $X_1$ and $X_2$ be  disjoint subsets of $\Finp$ with  $\min\{|X_1|,|X_2|\}\geq 3$ and $|X_1|+|X_2|\leq p-4$. If there are $x_1,y_1,r\in\Finp$ such that $X_1=r[x_1,y_1]$ and 
\begin{equation}
\label{E62}
|X_1-X_2\cup X_2-X_1|\leq |X_1|+|X_2|+1,
\end{equation}
then one of the following statements hold true:
\begin{enumerate}
 \item[(i)]There are $c,z,z'\in\Finp$ such that $X_2=r\big([y_1+c,x_1-c]\setminus \{z,z'\}\big)$.
 
 \item[(ii)]There are $x_2,y_2\in\Finp$ such that $X_2=r\big([x_2,y_2]\cup[x_1+y_1-y_2,x_1+y_1-x_2]\big)$ and $|X_1|=3$.
\end{enumerate} 
\end{lem}
\begin{proof}
Write $S:=X_1-X_2\cup X_2-X_1$. By Proposition \ref{R23} and (\ref{E62}), there are $x_2,x_3,x_4,y_2,y_3,y_4\in\Finp$ such that $X_2=\bigcup_{i=2}^4r[x_i,y_i]$ and $[x_i,y_i]\cap [x_j-1,y_j+1]=\emptyset$ for all $i\neq j$. Suppose without loss of generality that $[x_2,y_2]\neq \emptyset$. If $[x_3,y_3]=[x_4,y_4]=\emptyset$ or $\emptyset\not\in\{[x_3,y_3],[x_4,y_4]\}$, then (i) is implied by (\ref{E62}). If $\{\emptyset\}\subsetneq\{[x_3,y_3],[x_4,y_4]\}$, then it is checked straightforward  that $X_1$ and $X_2$ need to be as in (i) or (ii).
\end{proof}

\begin{lem}
\label{R16}
Let $X_1,X_2,X_3$ be pairwise disjoint subsets of $\Finp$ such that \\$\min\{|X_1|,|X_2|,|X_3|\}\geq 3$ and $\sum_{i=1}^3|X_i|\leq p-3$. Then
\begin{equation*}
\Bigg|\bigcup_{\{i,j,k\}=\{1,2,3\}}X_i+X_j-X_k\Bigg|>\sum_{i=1}^3|X_i|.
\end{equation*}
\end{lem}
\begin{proof}
Write $S:=\bigcup_{\{i,j,k\}=\{1,2,3\}}X_i+X_j-X_k$. We suppose that $|S|\leq \sum_{i=1}^3|X_i|$ and we shall arrive to a contradiction. Write
$S_i:=X_j-X_k\cup X_k-X_j$ for all $\{i,j,k\}=\{1,2,3\}$. Then 
\begin{equation}
\label{E21}
|S_i|+|X_i|-1\leq |S_i+X_i|\leq |S|\leq \sum_{j=1}^3|X_j|\qquad\forall\;i\in\{1,2,3\}.
\end{equation}
 We claim that there are $r,x,y\in\Finp$ such that $r[x,y]\in\{X_1,X_2,X_3\}$. Indeed, if $|S_1|\leq |X_2|+|X_3|$, the claim follows from Lemma \ref{R10}. If $|S_1|>|X_2|+|X_3|$, then Theorem \ref{R4} and (\ref{E21}) imply that $X_1=r[x,y]$ for some $r,x,y\in\Finp$. We assume without loss of generality that $X_1=r[x_1,y_1]$ for some $r,x_1,y_1\in\Finp$. Now we may apply Lemma \ref{R15} to $S_2$ and $S_3$ by (\ref{E21}); finally, it is easy to see that if $X_2$ and $X_3$ are as in (i) or (ii) of Lemma \ref{R15}, then $X_1$, $X_2$ and $X_3$ are not pairwise disjoint or $|S_1|>|X_2|+|X_3|+1$.
 \end{proof}
\section{Special cases with $n>3$}

\begin{lem}
\label{R17}
Assume that $n\geq 2$. Let  $x_1,\ldots, x_n,y_1,\ldots,y_n,a$ be elements of $\Finp$  such that $[x_1,y_1],\ldots,$ $[x_n,y_n]$ are pairwise disjoint, 
$3\leq \min_{1\leq i\leq n}|[x_i,y_i]|$ and $\sum_{i=1}^n|[x_i,y_i]|\leq p-1$. Write
\begin{equation*}
S:=\Bigg\{az_i+\sum_{j=1, i\neq j}^nx_j:\;i\in\{1,\ldots,n\}, z_i\in[x_i,y_i]\Bigg\}.
\end{equation*}
 If $x,y\in\Finp$  satisfy that $S\subseteq [x,y]$ and $a\not\in\{0,\pm 1\}$, then
\begin{equation}
\label{E24}
|[x,y]|>\max_{1\leq i\leq n}|[x_i,y_i]|+n-2.
\end{equation}
\end{lem}
\begin{proof}
Set $I:=[x,y]$ and  $M:=\max_{1\leq i\leq n}|[x_i,y_i]|$. We assume that there exist $x$ and $y$ such that (\ref{E24}) is not true, and we arrive to a contradiction.

First we show that $M\leq n-1$. Indeed, suppose without loss of generality that $M=|[x_1,y_1]|$; then 
\begin{equation}
\label{E64}
a[x_1,y_1]\subseteq S-\sum_{j=2}^nx_j\subseteq I-\sum_{j=2}^nx_j.
\end{equation}
Let $R$ be the element of $\{0,\ldots, p-1\}\subseteq \Zet$ such that $\overline{R}=a$ and we assume without loss of generality that $R<p-R$. Applying  Lemma \ref{R25} to (\ref{E64}), we obtain that 
\begin{align}
\label{E65}
n-1&\geq |I|-M+1\nonumber\\
&= \Bigg|\Bigg(I-\sum_{j=2}^nx_j\Bigg)\setminus a[x_1,y_1]\Bigg|+1\nonumber\\
&\geq \min\{R,M\}.
\end{align}
On one hand the assumptions $3\leq \min_{1\leq i\leq n}|[x_i,y_i]|$ and $\sum_{i=1}^n|[x_i,y_i]|\leq p-1$ yield
 \begin{equation}
 \label{E67}
 M+3(n-1)\leq p.
 \end{equation}
On the other hand if $M>n-1$, then $M\geq R$ by (\ref{E65}). Hence Lemma \ref{R25} leads to the inequality
\begin{align*}
M+n-2&\geq |I|\nonumber\\
&\geq\bigg[\frac{M}{R}\bigg]+(R-1)\bigg[\frac{p}{R}\bigg]\nonumber\\
&>\frac{M-R}{R}+(R-1)\bigg(\frac{p-R}{R}\bigg)
\end{align*}
 and consequently
 \begin{equation}
 \label{E66}
 M+(n-1)\bigg(\frac{R}{R-1}\bigg)+R>p.
 \end{equation}
Inasmuch as $R\geq 2$ (\ref{E66}) and (\ref{E67}) contradict (\ref{E65}) and therefore $M\leq n-1$.

 Define 
\begin{equation*}
S':=\Bigg\{ax_i+(a-1)\delta+\sum_{j=1, i\neq j}^nx_j:\;i\in\{1,\ldots,n\}, \delta\in\{0,1\}\Bigg\};
\end{equation*}
As $x_i\in\{x_j-1,x_j,x_j+1\}$ implies that $i=j$, we conclude that $|S'|=2n$. See that $ax_i+a+\sum_{j=1, i\neq j}^nx_j\in I$ for all $i\in\{1,\ldots,n\}$; then, since $I$ is an interval,  we have that $ax_i+(a-1)+\sum_{j=1, i\neq j}^nx_j\in I$ for all $i\in\{1,\ldots,n\}$ except at most one element. In particular 
\begin{equation}
\label{E1}
|S'\cap I|\geq 2n-1.
\end{equation}
We already know that $M\leq n-1$; then we obtain the following contradiction 
\begin{align*}
2n-1&\leq |S'\cap I|&\text{by (\ref{E1})}\\
&\leq |I|\\
&\leq M+n-2\\
&\leq 2n-3 .\end{align*}
 \end{proof}

\begin{lem}
\label{R18}
Let $n\geq 2$ and  $X_1,\ldots, X_n$ be pairwise disjoint subsets of $\Finp$ with $\min_{1\leq i\leq n}|X_i|\geq 3$ and $\sum_{i=1}^n|X_i|\leq p-5$. For $a\in\Finp^*\setminus\{\pm1\}$
\begin{equation}
\label{E22}
\Bigg|\bigcup_{i=1}^n\Bigg(aX_i+\sum_{j=1,j\neq i}^nX_j\Bigg)\Bigg|>\sum_{i=1}^n|X_i|.
\end{equation}
\end{lem}
\begin{proof}
We prove it by induction on $n$. If $n\in\{2,3\}$, then the result follows by Lemma \ref{R11} and  Lemma \ref{R13}. We assume that $n\geq 4$ and the result is true for all $m\leq n-1$. Write $S:=\bigcup_{i=1}^n\big(aX_i+\sum_{j=1,j\neq i}^nX_j\big)$  and 
\begin{equation*}
S_k:=\bigcup_{i=1,i\neq k}^n\Bigg(aX_i+\sum_{j=1,j\not\in\{k,i\}}^nX_j\Bigg)\qquad\forall k\in\{1,\ldots,n\}.
\end{equation*}
 We suppose that (\ref{E22}) is not true and we shall get a contradiction. For each $k\in\{1,\ldots,n\}$ 
\begin{align}
\sum_{i=1}^n|X_i|&\leq |S_k|+|X_k|-1&\text{by induction hypothesis}\nonumber\\
&\leq |S_k+X_k|&\text{by Theorem \ref{R3}}\nonumber\\
&=\Bigg| \bigcup_{i=1,i\neq k}^n\Bigg(aX_i+\sum_{j=1,j\neq i}^nX_j\Bigg)\Bigg|\nonumber\\
&\leq |S|\nonumber\\
&\leq \sum_{i=1}^n|X_i|\label{E23}
\end{align} 
and therefore all the inequalities of (\ref{E23}) are equalities. In particular, for all $k\in\{1,\ldots,n\}$, there are $x_k,y_k,x_k',y_k', r_k\in\Finp$ such that $X_k=r_k[x_k,y_k]$ and $S_k=r_k[x_k',y_k']$ by Theorem \ref{R4}. For all $k,k'\in\{1,\ldots,n\}$ with $k<k'$, define
\begin{equation*}
S_{k,k'}:=\bigcup_{i=1,i\not\in\{k,k'\}}^n\Bigg(aX_i+\sum_{j=1,j\not\in\{k,k',i\}}^nX_j\Bigg).
\end{equation*}
Then
\begin{align*}
\Bigg(\sum_{i=1}^n|X_i|\Bigg)-1&\leq |S_{k,k'}|+|X_k|+|X_{k'}|-2&\text{by induction hypothesis}\nonumber\\
&\leq |S_{k,k'}+X_k+X_{k'}|&\text{by Theorem \ref{R3}}\nonumber\\
&=\Bigg| \bigcup_{i=1,i\not\in\{k,k'\}}^n\Bigg(aX_i+\sum_{j=1,j\neq i}^nX_j\Bigg)\Bigg|\nonumber\\
&\leq |S|\nonumber\\
&\leq \sum_{i=1}^n|X_i|
\end{align*} 
and in particular $|X_k+X_{k'}|\leq |X_k|+|X_{k'}|$; then Lemma \ref{R12} yields that $r_k\in\{\pm r_{k'}\}$. As a consequence, we may assume without loss of generality that  $r_k=1$ for all $k\in\{1,\ldots,n\}$. For each $k\in\{1,\ldots,n\}$ and $z\in[x_k,y_k]$, define $S^{(k)}_{z}:=az+\sum_{i=1,i\neq k}^nX_i$, $x^{(k)}_{z}:=az+\sum_{i=1,i\neq k}^nx_i$ and $y^{(k)}_{z}:=az+\sum_{i=1,i\neq k}^ny_i$; thus $S^{(k)}_{z}=\big[ x^{(k)}_{z}, y^{(k)}_{z}\big]$ and 
\begin{equation}
\label{E63}
|S^{(k)}_{z}|=\Bigg|az+\sum_{i=1,i\neq k}^nX_i\Bigg|=\Bigg(\sum_{i=1,i\neq k}^n|X_i|\Bigg)-(n-2).
\end{equation}
However, by Lemma \ref{R17}, if $x,y\in\Finp$ are chosen such that
\begin{equation*}
 \big\{x^{(k)}_{z}:\;k\in\{1,\ldots,n\},\;z\in[x_k,y_k]\big\}\subseteq [x,y],
 \end{equation*} then $|[x,y]|> n-2+\max_{1\leq k\leq n}|X_k|$. Finally assume without loss of generality that $|X_1|=\max_{1\leq k\leq n}|X_k|$. By the above argumentation
\begin{align*}
|S|&=\Bigg|\bigcup_{1\leq k\leq n,\;z\in[x_k,y_k]}S^{(k)}_{z}\Bigg|\\
&> n-2+|X_1|+\Bigg(\Bigg(\sum_{i=2}^n|X_i|\Bigg)-(n-2)\Bigg)&\text{by (\ref{E63})}\\
&=\sum_{i=1}^n|X_i|
\end{align*}
and this contradicts our assumption.
\end{proof}
\begin{rem}
\label{R24}
If $n\geq 3$ in Lemma \ref{R18}, then the assumption $\sum_{i=1}^n|X_i|\leq p-3$ can be weakened  to $\sum_{i=1}^n|X_i|\leq p-3$ since the former assumption is just used in the small cases $n=2$.
\end{rem}
\begin{lem}
\label{R19}
Let $n\geq 3$ and   $X_1,\ldots, X_n$ be  pairwise disjoint subsets of $\Finp$ with $\min_{1\leq i\leq n}|X_i|\geq 3$ and $\sum_{i=1}^n|X_i|\leq p-3$. Then 
\begin{equation}
\label{E25}
\Bigg|\bigcup_{i=1}^n\Bigg(-X_i+\sum_{j=1,j\neq i}^nX_j\Bigg)\Bigg|>\sum_{i=1}^n|X_i|.
\end{equation}
\end{lem}
\begin{proof}
The proof is by induction on $n$. If $n=3$, then this is Lemma \ref{R16}. From now on, $n\geq 4$ and the result is true for $m\in\{3,\ldots, n-1\}$. Write
\begin{equation*}
S:=\bigcup_{i=1}^n\Bigg(-X_i+\sum_{j=1,j\neq i}^nX_j\Bigg)
\end{equation*}
 and
\begin{equation*}
S_k:=\bigcup_{i=1,i\neq k}^n\Bigg(-X_i+\sum_{j=1,j\not\in\{i,k\}}^nX_j\Bigg)\qquad\forall\;k\in\{1,\ldots,n\}.
\end{equation*} 
Assume  that (\ref{E25}) is false, and we shall arrive to a contradiction. See that
\begin{align}
\sum_{i=1}^n|X_i|&\geq|S|\nonumber\\
&\geq|S_k+X_k|\nonumber\\
&\geq|S_k|+|X_k|-1&\text{by Theorem \ref{R3}}\nonumber\\
&\geq\sum_{i=1}^n|X_i|&\text{by induction hypothesis}\label{E26}
\end{align}
so in (\ref{E26}) we have only equalities. Then, from Theorem \ref{R4},  there are $r_k, x_k, y_k,\\ r'_k , x'_k, y'_k \in\Finp$ such that $X_k=r_k[x_k,y_k]$ and $S_k=r_k[x_k',y_k']$ for all $k\in\{1,\ldots,n\}$.

Now we show that if $n=4$,  then $r_i\in\{\pm r_j\}$ for all $ i,j\in\{1,\ldots,4\}$. Indeed, assume without loss of generality that   $r_1\not\in\{\pm r_2\}$ and write $S_{1,2}:=X_1-X_2\cup X_2-X_1$ so 
\begin{equation}
\label{E27}
|S_{1,2}|\geq |r_1[x_1,y_1]-r_2[x_2,y_2]|> |X_1|+|X_2|
\end{equation}
by Lemma \ref{R12}. Then
\begin{align*}
\sum_{i=1}^4|X_i|&\geq|S|\\
&\geq |S_{1,2}+X_3+X_4|\\
&\geq |S_{1,2}|+|X_3|+|X_4|-2&\text{by Theorem \ref{R3}}\\
&\geq \sum_{i=1}^4|X_i|-1&\text{by (\ref{E27}),}
\end{align*}
and consequently $|X_3|+|X_4|\leq |X_3+X_4|$; thus $r_3\in\{\pm r_4\}$ by Lemma \ref{R12}. We have that either $r_3\not\in\{\pm r_2\}$ or  $r_3\not\in\{\pm r_1\}$; assume without loss of generality that $r_3\not\in\{\pm r_1\}$, then proceeding as above $r_2\in\{\pm r_4\}$. Thus for all $\{i,j,k\}=\{2,3,4\}$ there are $z_k,w_k\in\Finp$ such that $X_i+X_j-X_k:=r_2[z_k,w_k]$  and 
\begin{equation}
\label{E28}
|[z_2,w_2]|=|[z_3,w_3]|=|[z_4,w_4]|=\Bigg(\sum_{i=2}^4|X_i|\Bigg)-2.
\end{equation}
One one hand  $|S_1|=1+\sum_{i=2}^4|X_i|$ by (\ref{E26}); from (\ref{E28}) and the assumption $\min_{1\leq i\leq n}|X_i|\geq 3$, there are $z_1,w_1\in\Finp$ such that $S_1=r_2[z_1,w_1]$. On the other hand $S_1=r_1[x_1',y_1']$ so the assumption $r_1\not\in\{\pm r_2\}$ contradicts Proposition \ref{R7}.

 We show that $r_i\in\{\pm r_j\}$ for all $ i,j\in\{1,\ldots, n\}$ whether $n>4$. Call 
\begin{equation*}
S_{k,k'}:=\bigcup_{i=1,i\not\in\{k,k'\}}^n\Bigg(-X_i+\sum_{j=1,j\not\in\{k,k',i\}}^nX_j\Bigg)\qquad \forall\;k,k'\in\{1,\ldots,n\}\text{ with }k<k';
\end{equation*}
thus
\begin{align*}
\Bigg(\sum_{i=1}^n|X_i|\Bigg)-1&\leq |S_{k,k'}|+|X_k|+|X_{k'}|-2&\text{by induction hypothesis}\nonumber\\
&\leq |S_{k,k'}+X_k+X_{k'}|&\text{by Theorem \ref{R3}}\nonumber\\
&=\Bigg| \bigcup_{i=1,i\not\in\{k,k'\}}^n\Bigg(-X_i+\sum_{j=1,j\neq i}^nX_j\Bigg)\Bigg|\nonumber\\
&\leq |S|\nonumber\\
&\leq \sum_{i=1}^n|X_i|
\end{align*} 
and in particular $|X_k+X_{k'}|\leq |X_k|+|X_{k'}|$; then $r_k\in\{\pm r_{k'}\}$ by Lemma \ref{R12}.

We assume without loss of generality that $r_k=1$ for all $k\in\{1,\ldots,n\}$ from now on. Rearranging  $x_1,\ldots, x_n$, we may suppose that $[x_k,y_k]\subseteq [x_1,x_{k+1}]$  for all $k\in\{1,\ldots,n-1\}$. Set $S':=\Big\{-y_i+\sum_{k=1,k\neq i}^nx_k:\;i \in\{1,\ldots,n\}\Big\}$. If for some $ i_0,j_0\in\{1,\ldots,n\}$ with $i_0\neq j_0$ we get
\begin{equation*}
-y_{i_0}+\sum_{k=1,k\neq i_0}^nx_k=-y_{j_0}+\sum_{k=1,k\neq j_0}^nx_k,
\end{equation*}
then for all $k_0\not\in \{i_0,j_0\}$ and $\delta\in[-2,2]$
\begin{equation}
\label{E30}
-y_{k_0}+\sum_{k=1,k\neq k_0}^nx_k\neq \Bigg(-y_{i_0}+\sum_{k=1,k\neq i_0}^nx_k\Bigg)+\delta
\end{equation}
insomuch as $\min_{1\leq i\leq n}|X_i|\geq 3$ and $X_1,\ldots,X_n$ are pairwise disjoint.Call 
\begin{equation*}
S_1':=\Bigg\{-y_i+\sum_{k=1,k\neq i}^nx_k:\; \exists\ j\neq i \text{ such that }-y_i+\sum_{k=1,k\neq i}^nx_k=-y_j+\sum_{k=1,k\neq j}^nx_k\Bigg\}
\end{equation*}
and $
S_2':=S'\setminus S_1'$. If $x,y\in \Finp$ are such that $S'\subseteq [x,y]$,  then 
\begin{align}
\label{E29}
|[x,y]|&\geq 3|S_1'|+|S_2'|&\text{by (\ref{E30})}\nonumber\\
&\geq 2|S_1'|+|S_2'|\nonumber\\
&=n.
\end{align}
On the other hand
\begin{equation}
\label{E31}
\Bigg|-X_i+\sum_{k=1,k\neq i}^nX_k\Bigg|=\Bigg|\Bigg[-y_i+\sum_{k=1,k\neq i}^nx_k, -x_i+\sum_{k=1,k\neq i}^ny_k\Bigg]\Bigg|=\Bigg(\sum_{k=1}^n |X_k|\Bigg)-(n-1).
\end{equation}
Finally
\begin{align*}
\sum_{k=1}^n |X_k|&\geq |S|\\
&=\Bigg|\bigcup_{i=1}^n\Bigg(-X_i+\sum_{k=1,k\neq i}^nX_i\Bigg)\Bigg|\\
&\geq n+\Bigg(\sum_{k=1}^n |X_k|\Bigg)-(n-1)&\text{by (\ref{E29}) and (\ref{E31})}\\
&= \Bigg(\sum_{k=1}^n |X_k|\Bigg)+1
\end{align*}
which is impossible.
\end{proof}
\begin{lem}
\label{R20}
\item(i) Let $a_1,a_2,a_3,a_4\in\{\pm 1\}$ be not all equal and $X_1,\ldots, X_4$ pairwise disjoint subsets of $\Finp$ with $\min_{1\leq i\leq 4}|X_i|\geq 3$ and $\sum_{i=1}^4|X_i|\leq p-4$. Then
\begin{equation}
\label{E32}
\Bigg|\bigcup_{\sigma\in \Per_4}\sum_{i=1}^4a_{\sigma(i)}X_i\Bigg|> \sum_{i=1}^4|X_i|.
\end{equation}
\item(ii) Let $a_1,a_2,a_3,a_4,a_5\in\{\pm 1\}$ be not all equal and $X_1,\ldots, X_5$ pairwise disjoint subsets of $\Finp$ with $\min_{1\leq i\leq 5}|X_i|\geq 3$ and $\sum_{i=1}^5|X_i|\leq p-4$. Then
\begin{equation}
\label{E40}
\Bigg|\bigcup_{\sigma\in \Per_5}\sum_{i=1}^5 a_{\sigma(i)}X_i\Bigg|>\sum_{i=1}^5|X_i|.
\end{equation}
\end{lem}
\begin{proof}
To prove (i), it is enough to do the case $1=a_1=a_2=-a_3=-a_4$ by Lemma \ref{R19}. We assume that (\ref{E32}) is false and we arrive to a contradiction. As in the first part of Lemma \ref{R19}, we can reduce to the case $X_k=[x_k,y_k]$ with $x_k,y_k\in\Finp$ for $k\in\{1,\ldots,4\}$ (however, instead of using the induction step, we use Lemma \ref{R16}). Write 
$S_{1,2}:=X_1-X_2\cup X_2-X_1$ so
\begin{align*}
\sum_{i=1}^4|X_i|&\geq \Bigg|\bigcup_{\sigma\in \Per_4}\sum_{i=1}^4a_{\sigma(i)}X_i\Bigg|\\
&\geq|S_{1,2}+X_3-X_4|\\
&\geq|S_{1,2}|+|X_3|+|X_4|-2&\text{by Theorem \ref{R3};}
\end{align*}
then $|S_{1,2}|\leq |X_1|+|X_2|+2$  and thereby there are $b_2,c_2\in\Finp$ with $c_2\in[b_2-3,b_2+3]$ such that $X_2=[y_1+b_2,x_1-c_2]$.  In the same way, there are $b_3,b_4,c_3,c_4\in\Finp$ such that
$X_3=[y_1+b_3,x_1-c_3]$ and $X_4=[y_1+b_4,x_1-c_4]$ with $c_3\in[b_3-3,b_3+3]$ and $c_4\in[b_4-3,b_4+3]$; this contradicts the pairwise disjointedness of $X_2,X_3,$ and $X_4$.

To show (ii), it is enough to do the case $1=a_1=a_2=a_3=-a_4=-a_5$ by Lemma \ref{R19}. We assume that (\ref{E40}) is false and we get a contradiction. As in the first part of Lemma \ref{R19} (however instead of using the induction step, we use Lemma \ref{R20} (i)), we can reduce to the case $X_k=[x_k,y_k]$ with $x_k,y_k\in\Finp$ for $k\in\{1,\ldots,5\}$. Call $S_{1,2}:=X_1-X_2\cup X_2-X_1$ and we deduce that $|S_{1,2}|\leq |X_1|+|X_2|+3$  with the same analysis as in (i).  This means that there are $b_2',c_2'\in\Finp$ such that $X_2=[y_1+b_2',x_1-c_2']$ with $c_2'\in[b_2'-4,b_2'+4]$. In the same way, there are $b_3',b_4',b_5',c_3',c_4',c_5'\in\Finp$ such that $X_i=[y_1+b_i',x_1-c_i']$
with $c_i'\in[b_i'-4,b_i'+4]$ for  all $i\in\{3,4,5\}$. Then $X_1,X_2,X_3,X_4$ and $X_5$ are not disjoint.
\end{proof}

\begin{lem}
\label{R21}
Let $X_1,\ldots,X_4\subseteq \Finp$ be pairwise disjoint subsets  with $\min_{1\leq i\leq 4}|X_i|  \\\geq 3$ and $\sum_{i=1}^4|X_i|\leq p-4.$ If $a_1,a_2,a_3,a_4$ are elements of $\Finp^*$ such that   $a_1=a_2=-a_3$, then 
\begin{equation}
\label{E37}
\Bigg|\bigcup_{\sigma\in \Per_4}\sum_{i=1}^4a_{\sigma(i)}X_i\Bigg|>\sum_{i=1}^4|X_i|.
\end{equation}
\end{lem}
\begin{proof}
From Lemma \ref{R20} we may assume that $a_4\not\in\{\pm a_1\}$. We arrive to a contradiction whether  (\ref{E37}) is false.  Write
\begin{equation*}
S:=\bigcup_{\sigma\in \Per_4}\sum_{i=1}^4a_{\sigma(i)}X_i\qquad\text{and}\qquad S_4:=\bigcup_{\sigma\in \Per_3}\sum_{i=1}^3a_{\sigma(i)}X_i.
\end{equation*} 
Then
\begin{align}
\sum_{i=1}^4|X_i|&\geq |S|\nonumber\\
&\geq|S_4+a_4X_4|\nonumber\\
&\geq|S_4|+|X_4|-1&\text{by Theorem \ref{R3}}\nonumber\\
&\geq\sum_{i=1}^4|X_i|&\text{by Lemma \ref{R16};}\label{E38}
\end{align}
Thus all the relations in (\ref{E38}) are equalities. By Theorem \ref{R4} there are $r_4,x_4,y_4,\\x_4',y_4' \in\Finp$ such that $a_4X_4=r_4[x_4,y_4]$ and $S_4=r_4[x_4',y_4']$. Analogously  there are $r_i,x_i,y_i,x_i',y_i'\in\Finp$
such that $a_4X_i=r_i[x_i,y_i]$ for all $i\in\{1,2,3\}$. Call $S_{3,4}:=a_3X_3+a_4X_4\cup a_3X_4+a_4X_3$ and note that
\begin{align}
\sum_{i=1}^4|X_i|&\geq |S|\nonumber\\
&\geq|a_1X_1+a_2X_2+S_{3,4}|\nonumber\\
&\geq|X_1|+|X_2|+|S_{3,4}|-2&\text{by Theorem \ref{R3}}\nonumber\\
&\geq\Bigg(\sum_{i=1}^4|X_i|\Bigg)-1&\text{by Lemma \ref{R11}}\nonumber
\end{align} 
thus $|X_1|+|X_2|\geq |a_1X_1+a_2X_2|$, and $r_1\in\{\pm r_2\}$ by Lemma \ref{R12}. In the same way, it can be proven that $r_i\in\{\pm r_j\}$ for all $i,j\in\{1,\ldots,4\}$. Assume without loss of generality that $r_i=a_4$ for all $i\in\{1,\ldots,4\}$ and call $S'_{1,2}:=a_1X_1+a_3X_2\cup a_1X_2+a_3X_1$. Then 
\begin{align}
\sum_{i=1}^4|X_i|&\geq |S|\nonumber\\
&\geq|S_{1,2}'+a_2X_3+a_4X_4|\nonumber\\
&\geq|S_{1,2}'|+|X_3|+|X_4|&\text{by Lemma \ref{R12}}\nonumber\\
&\geq\Bigg(\sum_{i=1}^4|X_i|\Bigg)-1&\text{by Theorem \ref{R3}.}\label{E41}
\end{align}
Hence (\ref{E41}) states that  $|S'_{1,2}|\leq |X_1|+|X_2|+1$ and thereby there are $b_2,c_2\in\Finp$ such that $X_2=[y_1+b_2,x_1-c_2]$ with $c_2\in [b_2-2,b_2+2]$. Proceeding as above, there are $b_3,b_4,c_3,c_4\in\Finp$ such that $X_3=[y_1+b_3,x_1-c_3]$ and $X_4=[y_1+b_4,x_1-c_4]$ with $c_3\in [b_3-2,b_3+2]$ and $c_4\in [b_4-2,b_4+2]$; thus $X_1,X_2,X_3$ and $X_4$ are not pairwise disjoint.
\end{proof}
\section{Proof of Theorem 1.1}
In this section we prove Theorem \ref{R1}. Assume without loss of generality that the $a_1,\ldots,a_n$ are ordered such that  there exist  $1\leq k_1<k_2<\ldots<k_m=n$ with  $a_1=a_2=\ldots=a_{k_1}$,  $a_{k_i+1}=a_{k_i+2}=\ldots=a_{k_{i+1}}$ for all $i\in\{1,\ldots,m-1\}$ and with $a_{k_i}=a_{k_j}$ only if $i=j$. 
\begin{proof}\emph{(Theorem \ref{R1})} The proof is by induction on $n$. The result follows from Lemma \ref{R10} and Lemma \ref{R11} when $n=2$. Also the result follows from Lemma \ref{R13} and Lemma \ref{R16} when $n=3$. From now on $n\geq 4$ and we assume that the result is true for all $ n'\in\{2,\ldots,n-1\}$. The induction step  depends on  $m$ and we analyze the following cases: 
\begin{enumerate}
\item[]Suppose that  $m\geq 4$. Then we can find a partition $A_1\cup A_2$ of $\{a_{k_1},\ldots,a_{k_m}\}$ such that $\min\{|A_1|,|A_2|\}>1$ and there are $b_i,c_i\in A_i$ such that $b_i\not\in\{\pm c_i\}$ for $i\in\{1,2\}$. Assume without loss of generality that $A_1=\{a_{k_1},a_{k_2}\}$ and $A_2=\{a_{k_3},\ldots,a_{k_m}\}$. Set
\begin{equation*}
S_1:=\bigcup_{\sigma\in \Per_{k_2}}\sum_{i=1}^{k_2}a_{\sigma(i)}X_i\qquad\text{and}\qquad S_2:=\bigcup_{\sigma\in \Per_{n-k_2}}\sum_{i=1}^{n-k_2}a_{\sigma(i)+k_2}X_{i+k_2};
\end{equation*}
 then 
 \begin{align}
 \label{E2}
  |S|&\geq|S_1+S_2|\nonumber\\
  &\geq|S_1|+|S_2|-1&\text{by Theorem \ref{R3}}\nonumber\\
  &\geq \Bigg(\sum_{i=1}^{k_2}|X_i|\Bigg)+1+\Bigg(\sum_{i=k_2+1}^n|X_i|\Bigg)+1-1&\text{by induction}\nonumber\\
  &>\Bigg(\sum_{i=1}^n|X_i|\Bigg).
 \end{align}
 
 Until the end of the proof, we assume without loss of generality that $k_1\geq k_2-k_1\geq\ldots\geq k_m-k_{m-1}$.
 
\item[]Suppose that $m=3$. First we deal with the case $a_{k_1}\neq -a_{k_2}$. Write
\begin{equation*}
a_i'= \left\{ \begin{array}{lll}
a_{k_1} &\text{if }i=k_1+1\\
a_{k_1+1} &\text{if }i=k_1 \\
a_i &\text{otherwise}\end{array}, \right.
\end{equation*}
  \begin{equation*}
S_1:=\bigcup_{\sigma\in \Per_{k_1}}\sum_{i=1}^{k_1}a'_{\sigma(i)}X_i\qquad\text{and}\qquad S_2:=\bigcup_{\sigma\in \Per_{n-k_1}}\sum_{i=1}^{n-k_1}a'_{\sigma(i)+k_1}X_{i+k_1} 
\end{equation*} 
and we conclude  as in (\ref{E2}).  Now assume that $a_{k_1}=-a_{k_2}$ and $k_2-k_1>1$. In this case we set
\begin{equation*}
a_i'= \left\{ \begin{array}{lll}
a_{k_1} &\text{if }i=k_2+1\\
a_{k_2+1} &\text{if }i=k_1 \\
a_i &\text{otherwise}\end{array}, \right.
\end{equation*}
  \begin{equation*}
S_1:=\bigcup_{\sigma\in \Per_{k_1}}\sum_{i=1}^{k_1}a'_{\sigma(i)}X_i\qquad\text{and}\qquad S_2:=\bigcup_{\sigma\in \Per_{n-k_1}}\sum_{i=1}^{n-k_1}a'_{\sigma(i)+k_1}X_{i+k_1},
\end{equation*} 
 and we proceed as in (\ref{E2}).  If  $a_{k_1}=-a_{k_2}$ and $k_2-k_1=1$, then $k_3-k_2=1$  since $k_3-k_2\leq k_2-k_1$. Insomuch as $n\geq 4$,  we get that $k_1\geq 2$; moreover, we may assume that $k_1>2$ by Lemma \ref{R21}. Defining
\begin{equation*}
a_i'= \left\{ \begin{array}{lll}
a_{k_1} &\text{if }i=k_1+1\\
a_{k_1+1} &\text{if }i=k_1 \\
a_i &\text{otherwise}\end{array}, \right.
\end{equation*}
\begin{equation*}
S_1:=\bigcup_{\sigma\in \Per_{k_1}}\sum_{i=1}^{k_1}a'_{\sigma(i)}X_i\qquad\text{and}\qquad S_2:=\bigcup_{\sigma\in \Per_{n-k_1}}\sum_{i=1}^{n-k_1}a'_{\sigma(i)+k_1}X_{i+k_1},
\end{equation*} 
 we obtain the result  concluding as in (\ref{E2}).
 
\item[]Suppose that $m=2$. By Lemma  \ref{R18} and Lemma \ref{R19}, it suffices to solve the case  $k_2-k_1>1$. If $a_{k_1}\neq -a_{k_2}$, then define
\begin{equation*}
a_i'= \left\{ \begin{array}{lll}
a_{k_1} &\text{if }i=k_1+1\\
a_{k_1+1} &\text{if }i=k_1 \\
a_i &\text{otherwise}\end{array}, \right.
\end{equation*}
\begin{equation*}
S_1:=\bigcup_{\sigma\in \Per_{k_1}}\sum_{i=1}^{k_1}a'_{\sigma(i)}X_i\qquad\text{and}\qquad S_2:=\bigcup_{\sigma\in \Per_{n-k_1}}\sum_{i=1}^{n-k_1}a'_{\sigma(i)+k_1}X_{i+k_1},
\end{equation*} 
and we finish as in (\ref{E2}).   If $a_{k_1}=-a_{k_2}$ and  $k_2-k_1=2$, then $k_1\geq 2$. By Lemma \ref{R20}, it is enough to demonstrate the claim when $k_1\geq 4$. We may conclude as in (\ref{E2}) defining 
\begin{equation*}
a_i'= \left\{ \begin{array}{lll|}
a_{k_1+1} &\text{if }i=k_1-1\\
a_{k_1-1} &\text{if }i=k_1 \\
a_{k_1} &\text{if }i=k_1+1 \\
a_i &\text{otherwise}\end{array} \right.
\end{equation*}
\begin{equation*}
S_1:=\bigcup_{\sigma\in \Per_{k_1-1}}\sum_{i=1}^{k_1-1}a'_{\sigma(i)}X_i\quad\text{and}\quad S_2:=\bigcup_{\sigma\in \Per_{n-k_1+1}}\sum_{i=1}^{n-k_1+1}a'_{\sigma(i)+k_1-1}X_{i+k_1-1}.
\end{equation*} 
Finally, if $a_{k_1}=-a_{k_2}$ and  $k_2-k_1>2$, then define
\begin{equation*}
a_i'= \left\{ \begin{array}{lll}
a_{k_1} &\text{if }i=k_1+1\\
a_{k_1+1} &\text{if }i=k_1 \\
a_i &\text{otherwise}\end{array}, \right.
\end{equation*}
\begin{equation*}
S_1:=\bigcup_{\sigma\in \Per_{k_1}}\sum_{i=1}^{k_1}a'_{\sigma(i)}X_i\qquad\text{and}\qquad S_2:=\bigcup_{\sigma\in \Per_{n-k_1}}\sum_{i=1}^{n-k_1}a'_{\sigma(i)+k_1}X_{i+k_1},
\end{equation*} 
and the result follows as in (\ref{E2}).
\end{enumerate}
\end{proof}
\section{Proof of Theorem 1.2}
In this section we show Theorem 1.2. As in the proof of Theorem \ref{R1}, assume without loss of generality that there are $1\leq k_1<k_2<\ldots<k_m=n$ such that $a_1=a_2=\ldots=a_{k_1}$ and $a_{k_i+1}=a_{k_i+2}=\ldots=a_{k_{i+1}}$ for all $i\in\{1,\ldots,m-1\}$ with $a_{k_i}=a_{k_j}$ only if $i=j$. The main idea that we will use in the proof is that if there are not rainbow solutions of (\ref{E48}), then $b\not\in\bigcup_{\sigma\in \Per_n}\sum_{i=1}^na_{\sigma(i)}C_i=:S$. Thus to show that (\ref{E48}) has a rainbow solution, it is enough to prove the following inequality
\begin{equation}
\label{E47}
|S|>|\Finp\setminus\{b\}|=p-1.
\end{equation}
\begin{proof}\emph{(Theorem \ref{R2})} First assume that $n=2$. If $C_1$ and $C_2$ are as in (\ref{E45}), then
\begin{equation*}
a_1C_1=-a_2C_1+b\qquad\text{and}\qquad a_1C_2=-a_2C_2+b,
\end{equation*}
and the result is clear. If the coloring is rainbow free with respect to (\ref{E48}), then 
\begin{equation*}
a_1C_1\cap (-a_2C_2+b)=\emptyset\qquad\text{and}\qquad a_1C_2\cap (-a_2C_1+b)=\emptyset
\end{equation*} 
which is equivalent to say that
\begin{equation*}
a_1C_1=-a_2C_1+b\qquad\text{and}\qquad a_1C_2=-a_2C_2+b;
\end{equation*}
then $C_1$ and $C_2$ have to be as in (\ref{E45}).

Due to the main result of [6] and the previous paragraph, we may assume that $n>3$. We shall show (\ref{E47}) studying the possibilities of $m$:
\begin{enumerate}
\item[]Suppose that $m\geq 4$. Then we can find a partition  $A_1\cup A_2$ of $\{a_{k_1},\ldots,a_{k_m}\}$ with the properties that $\min\{|A_1|,|A_2|\}>1$ and there are $b_i,c_i\in A_i$ such that $b_i\not\in\{\pm c_i\}$ for $i\in\{1,2\}$. Assume without loss of generality that $A_1=\{a_{k_1},a_{k_2}\}$ and $A_2=\{a_{k_3},\ldots,a_{k_m}\}$. Call 
\begin{equation*}
S_1:=\bigcup_{\sigma\in \Per_{k_2}}\sum_{i=1}^{k_2}a_{\sigma(i)}C_i\qquad\text{and}\qquad S_2:=\bigcup_{\sigma\in \Per_{n-k_2}}\sum_{i=1}^{n-k_2}a_{\sigma(i)+k_2}C_{i+k_2}.
\end{equation*}
If (\ref{E47}) is not true, then
\begin{equation*}
\max\Bigg\{\sum_{i=1}^{k_2}|C_i|, \sum_{i=k_2+1}^n|C_i| \Bigg\}\leq p-5
\end{equation*}
and 
\begin{align}
\label{E4}
p-1&\geq |S|\nonumber\\
&\geq |S_1+S_2|\nonumber\\
&\geq |S_1|+|S_2|-1&\text{by Theorem \ref{R3}}\nonumber\\
&\geq \Bigg(\sum_{i=1}^{k_2}|C_i|\Bigg)+1+\Bigg(\sum_{i=k_2+1}^n|C_i|\Bigg)+1-1&\text{ by Theorem \ref{R1} }\nonumber\\
&=p+1
\end{align}
which is false.

Until the end of this proof, we suppose without loss of generality that $k_1\geq k_2-k_1\geq\ldots\geq k_m-k_{m-1}$

\item[]Suppose that  $m=3$. If $a_{k_1}\neq-a_{k_2}$, write 
\begin{equation*}
a_i'= \left\{ \begin{array}{lll}
a_{k_1} &\text{if }i=k_1+1\\
a_{k_1+1} &\text{if }i=k_1 \\
a_i &\text{otherwise}\end{array}, \right.
\end{equation*}
  \begin{equation*}
S_1:=\bigcup_{\sigma\in \Per_{k_1}}\sum_{i=1}^{k_1}a'_{\sigma(i)}C_i\qquad\text{and}\qquad S_2:=\bigcup_{\sigma\in \Per_{n-k_1}}\sum_{i=1}^{n-k_1}a'_{\sigma(i)+k_1}C_{i+k_1};
\end{equation*} 
and  conclude as in (\ref{E4}).  If $a_{k_1}=-a_{k_2}$ and $k_2-k_1>1$, we set 
\begin{equation*}
a_i'= \left\{ \begin{array}{lll}
a_{k_1} &\text{if }i=k_2+1\\
a_{k_2+1} &\text{if }i=k_1 \\
a_i &\text{otherwise}\end{array}, \right.
\end{equation*}
  \begin{equation*}
S_1:=\bigcup_{\sigma\in \Per_{k_1}}\sum_{i=1}^{k_1}a'_{\sigma(i)}C_i\qquad\text{and}\qquad S_2:=\bigcup_{\sigma\in \Per_{n-k_1}}\sum_{i=1}^{n-k_1}a'_{\sigma(i)+k_1}C_{i+k_1}
\end{equation*} 
and  conclude as in (\ref{E4}).  If $a_{k_1}=-a_{k_2}$ and $k_2-k_1=1$, then $k_3-k_2=1$ and thereby $k_1\geq 2$. If $k_1>2$, then we conclude as in (\ref{E4}) taking
\begin{equation*}
a_i'= \left\{ \begin{array}{lll}
a_{k_1} &\text{if }i=k_1+1\\
a_{k_1+1} &\text{if }i=k_1 \\
a_i &\text{otherwise}\end{array}, \right.
\end{equation*}
  \begin{equation*}
S_1:=\bigcup_{\sigma\in \Per_{k_1}}\sum_{i=1}^{k_1}a'_{\sigma(i)}C_i\qquad\text{and}\qquad S_2:=\bigcup_{\sigma\in \Per_{n-k_1}}\sum_{i=1}^{n-k_1}a'_{\sigma(i)+k_1}C_{i+k_1}.
\end{equation*}
Now we study the case where $a_{k_1}=-a_{k_2}$, $k_2-k_1=k_3-k_2=1$ and $k_1=2$. Set 
$R_{i,j}:=a_1C_i+a_4C_j\cup a_1C_j+a_4C_i$  and $T_{i,j}:=a_2C_i+a_3C_j\cup a_2C_j+a_3C_i$ for each $i,j\in\{1,\ldots, 4\}$ with $i<j$. If (\ref{E47}) is not true, then 
\begin{align}
p-1&\geq |S|\nonumber\\
&\geq |R_{1,2}+T_{3,4}|\nonumber\\
&\geq |R_{1,2}|+\Bigg(\sum_{i=3}^4|C_i|\Bigg)-2&\text{by Theorem \ref{R3}}\nonumber\\
&\geq \Bigg(\sum_{i=1}^{2}|C_i|\Bigg)+1+\Bigg(\sum_{i=3}^4|C_i|\Bigg)-2&\text{ by Theorem \ref{R1} }\nonumber\\
&=p-1\nonumber
\end{align}
so $|T_{3,4}|=|C_3|+|C_4|-1$. As a consequence of Theorem \ref{R1}, there are $r,x,y,c\in\Finp$ such that $C_3=r[x,y]$ and $C_4=r[y+c,x-c]$. In the same way, it can be proven there are $r',x',y',c'\in\Finp$ such that $C_3=r'[x',y']$ and $ C_2=r'[y'+c',x'-c']$. By Proposition \ref{R7} we get that $r'\in\{\pm r\}$; we assume without loss of generality that $r'=r$ and thereby $x'=x$ and $y'=y$. Consequently $C_1, C_2, C_3$ and $C_4$ are not pairwise disjoint.

\item[] Suppose that $m=2$. In the case where $a_{k_1}\neq -a_{k_2}$ and  $k_2-k_1>1$ or in the case where $a_{k_1}=-a_{k_2}$ and  $k_2-k_1>2$, we conclude as in (\ref{E4}) with
\begin{equation*}
a_i'= \left\{ \begin{array}{lll}
a_{k_1} &\text{if }i=k_1+1\\
a_{k_1+1} &\text{if }i=k_1 \\
a_i &\text{otherwise}\end{array}, \right.
\end{equation*}
  \begin{equation*}
S_1:=\bigcup_{\sigma\in \Per_{k_1}}\sum_{i=1}^{k_1}a'_{\sigma(i)}C_i\qquad\text{and}\qquad S_2:=\bigcup_{\sigma\in \Per_{n-k_1}}\sum_{i=1}^{n-k_1}a'_{\sigma(i)+k_1}C_{i+k_1}.
\end{equation*} 
 If $a_{k_1}=-a_{k_2}$ and $(k_1,k_2-k_1)=(2,2)$,  write $R_{i,j}:=a_1C_i+a_3C_j\cup a_1C_j+a_3C_i$ and $T_{i,j}=a_2C_i+a_4C_j\cup a_2C_j+a_4C_i$ for $i,j\in\{1,\ldots,4\}$ with $i<j$. Then
\begin{align}
p-1&\geq |S|\nonumber\\
&\geq |R_{1,2}+T_{3,4}|\nonumber\\
&\geq |R_{1,2}|+|T_{3,4}|-1&\text{by Theorem \ref{R3}.}\nonumber
\end{align}
Hence $|R_{1,2}|\leq |C_1|+|C_2|-1$ or $|T_{3,4}|\leq |C_3|+|C_4|-1$; assume without loss of generality that $|R_{1,2}|\leq |C_1|+|C_2|-1$. Thus there are $r,x,y,c$ such that $C_1=r[x,y]$ and $C_2=r[y+c,x-c]$. Analogously  $|R_{1,3}|\leq |C_1|+|C_3|-1$ or $|T_{2,4}|\leq |C_2|+|C_4|-1$, and we assume without loss of generality that  $|R_{1,3}|\leq |C_1|+|C_3|-1$ so that there are $r',x',y',c'$ such that
$C_1=r'[x',y']$ and $C_3=r'[y'+c',x'-c']$. By Proposition \ref{R7} we conclude that $r'\in\{\pm r\}$; we suppose without loss of generality $r'=r$ so $x'=x$, $y'=y$, and $C_1,C_2,C_3$ are not pairwise disjoint. Now we study the case where $a_{k_1}=-a_{k_2}$, $k_1>2$ and $k_2-k_1=2$. Write
\begin{equation*}
a_i'= \left\{ \begin{array}{lll}
a_{k_1} &\text{if }i=k_1+1\\
a_{k_1+1} &\text{if }i=k_1 \\
a_i &\text{otherwise}\end{array}, \right.
\end{equation*}
\begin{equation*}
S_1:=\bigcup_{\sigma\in \Per_{k_1}}\sum_{i=1}^{k_1}a'_{\sigma(i)}C_i\qquad\text{and}\qquad S_2:=\bigcup_{\sigma\in \Per_{n-k_1}}\sum_{i=1}^{n-k_1}a'_{\sigma(i)+k_1}C_{i+k_1}.
\end{equation*}
If (\ref{E47}) is not true, then 
\begin{align*}
p-1&\geq |S|\\
&\geq |S_1+S_2|\\
&\geq |S_1|+|S_2|-1&\text{by Theorem \ref{R3}}\\
&\geq |S_1|+|C_{k_1+1}|+|C_{k_1+2}|-2&\text{by Theorem \ref{R3}}\\
&\geq \Bigg(\sum_{i=1}^{k_2}|C_i|\Bigg)+1-2&\text{ by Theorem \ref{R1} }\\
&=p-1
\end{align*}
and all these relations are equalities; in particular,  $|S_2|=|C_{k_1+1}|+|C_{k_1+2}|-1$. Then Theorem \ref{R1} impliesthe existence of $r,x,y,c\in\Finp$ such that
$C_{k_1+2}=r[x,y]$ and $C_{k_1+1}=r[y+c,x-c]$. Define 
\begin{equation*}
S_1':=\bigcup_{\sigma\in \Per_{k_1}}\Bigg(a'_{\sigma(k_1)}C_{k_1+1}+\sum_{i=1}^{k_1-1}a'_{\sigma(i)}C_i\Bigg)
\end{equation*}
and 
\begin{equation*}
 S_2':= a'_{k_1+1}C_{k_1}+a'_{k_1+2}C_{k_1+2}\cup a'_{k_1+1}C_{k_1+2}+a'_{k_1+2}C_{k_1};
\end{equation*}
 If we proceed as above  (with $(S_1',S_2')$ instead of $(S_1,S_2)$), we may obtain the existence of $r',x',y',c'\in\Finp$ such that $C_{k_1+2}=r'[x',y']$ and $ C_{k_1}=r'[y'+c',x'-c']$. From Proposition \ref{R7} we deduce that  $r'\in\{\pm r\}$; suppose without loss of generality that $r'=r$. Then $x'=x$, $y'=y$,  and $C_{k_1}, C_{k_1+1}, C_{k_1+1}$ are not pairwise disjoint.  Finally we analyze the case $k_1-k_2=1$. Set 
\begin{equation*}
a_i'= \left\{ \begin{array}{lll}
a_{k_1} &\text{if }i=k_1+1\\
a_{k_1+1} &\text{if }i=k_1 \\
a_i &\text{otherwise}\end{array} \right.
\end{equation*} 
and 
\begin{equation*}
S':=\bigcup_{\sigma\in \Per_{k_1}}\sum_{i=1}^{k_1}a'_{\sigma(i)}C_i;
\end{equation*}
If (\ref{E47}) is not true, we get the contradiction 
\begin{align*}
p-1&\geq |S|\\
&\geq |S'+a_{k_1+1}'C_{k_1+1}|\\
&\geq |S'|+|C_{k_1+1}|-1&\text{by Theorem \ref{R3}}\\
&\geq \Bigg(\sum_{i=1}^{k_1}|C_i|\Bigg)+1+|C_{k_1+1}|-1&\text{ by Lemma \ref{R18}, Lemma \ref{R19} }\\
&&\text{and Remark \ref{R24} }\\
&=p.
\end{align*}
\end{enumerate}
\end{proof}
\textbf{Acknowledgments: }
I acknowledge Amanda Montejano who introduced me to the topic and proposed this problem.

\end{document}